\documentclass[12pt]{article}

\usepackage{amsfonts,amssymb,amsmath,amsthm,enumerate,
dsfont,xcolor,mathtools,hyperref,babel,authblk,todonotes}
\usepackage[figurename=Figure]{caption}
\usepackage[mathscr]{euscript}
\usepackage[skins]{tcolorbox}
\usepackage[cp1251]{inputenc}
\usepackage{subfigure,xcolor}
\usepackage{fullpage}
\hypersetup{colorlinks,linkcolor={red!50!black},
	citecolor={blue!50!black}}

\newtheorem{theorem}{Theorem}[section]
\newtheorem{lemma}[theorem]{Lemma}
\newtheorem{proposition}[theorem]{Proposition}
\theoremstyle{definition}
\newtheorem{definition}[theorem]{Definition}
\newtheorem{example}[theorem]{Example}
\theoremstyle{remark}
\newtheorem{remark}[theorem]{Remark}
\newtheorem{openp}[theorem]{Problem}

\newcommand{\C}{\mathscr C}
\newcommand{\Kd}{\mathscr K^d}
\newcommand{\Ktwo}{\mathscr K^2}
\newcommand{\K}{\mathscr K^1}
\newcommand{\Rd}{\mathbb R^d}
\newcommand{\Rtwo}{\mathbb R^2}
\newcommand{\R}{\mathbb R}
\newcommand{\N}{\mathbb N}
\newcommand{\Ne}{\overline{\N}_0} 
\newcommand{\NC}{\mathcal N}
\newcommand{\F}{\mathcal F}
\newcommand{\Z}{\mathbb Z}
\newcommand{\Gd}{\mathbb G^d}
\newcommand{\Sd}{\mathbb S^{d-1}}
\newcommand{\Vd}{\mathbb V^d}

\newcommand{\1}{\mathds 1}
\newcommand{\eps}{\varepsilon}
\newcommand{\mchi}{\mathop{\chi}}
\newcommand{\hs}{\hspace{1pt}}
\newcommand{\Hr}{\mathop{\overset{\smash{\lower0.3ex\hbox{$\,\,
\scriptstyle\circ$}}}{H\!}}{}}
\renewcommand{\phi}{\varphi}
\renewcommand{\d}{\,\mathrm d}
\renewcommand{\emptyset}{\varnothing}

\DeclareMathOperator{\card}{card}
\DeclareMathOperator{\inter}{int}
\DeclareMathOperator{\conv}{conv}

\begin{document}

\title{Integer-valued valuations}

\author[1,2]{Andrii Ilienko}
\author[1]{Ilya Molchanov}
\author[1]{Tommaso Vison\`a}

\affil[1]{University of Bern}
\affil[2]{Igor Sikorsky Kyiv Polytechnic Institute}

\maketitle

\begin{abstract}
  We obtain a complete characterization of planar monotone
  $\sigma$-continuous valuations taking integer values, without
  assuming invariance under any group of transformations. We further
  investigate the consequences of dropping monotonicity or
  $\sigma$-continuity and give a full classification of line
  valuations. We also introduce a construction of the product for
  valuations of this type.
\end{abstract}

\\[6pt]

\section{Introduction}
\label{sec:introduction}

A valuation $\phi$ is an additive map from the family of compact
convex subsets of a finite-dimensional vector space to an abelian
semigroup. Additivity means that, for any compact convex sets $K$ and
$L$ such that $K\cup L$ is also convex, the following identity holds:
\[
\phi(K \cup L) + \phi(K \cap L) = \phi(K) + \phi(L),
\]  
see \cite[Chapter~6]{S14} for a detailed exposition. Additionally, we
will always include the empty set in the domain of $\phi$ and assume
that $\phi(\emptyset)=0$. Most of the literature on valuations focuses
on valuations with values in the set of real or complex numbers or in
the family of compact convex sets equipped with Minkowski addition.

A common assumption in the study of valuations is their invariance
under a group of transformations. In most cases, valuations are
assumed to be translation invariant, meaning that
$\phi(K + x) = \phi(K)$ for all translations $x$. Alternatively,
valuations are also studied under the assumption of rotation
invariance or invariance under the group of all rigid motions.
Another frequently imposed condition is continuity with respect to the
Hausdorff metric on compact convex sets. This condition is sometimes
relaxed to $\sigma$-continuity, which requires that
$\phi(K_n) \to \phi(K)$ whenever $K_n \downarrow K$.

Let $\Kd$ be the family of convex bodies (i.e., compact convex sets)
in $\R^d$. While the empty set is typically not considered a convex
body, we adopt the convention that it is included in $\Kd$.  By
Hadwiger's theorem, any real-valued continuous and 
invariant under rigid motions valuation on $\Kd$ can be expressed as a
weighted sum of the intrinsic volumes $V_i(K)$, $i = 0, \dots, d$.
Furthermore, McMullen's theorem states that the vector space of all
continuous translation-invariant valuations can be decomposed into a
direct sum of subspaces consisting of valuations that are homogeneous
of order $k=0, \dots, d$.
A more refined result is given by the theorem of Klain and
Schneider. It states that if $\phi$ is a continuous simple
translation-invariant valuation, then
\[
\phi(K) = c V_d(K) + \int_{\Sd} f(u)\d S_{d-1}(K,u),
\]  
where $c \in \mathbb{R}$, $f$ is an odd continuous function, and
$S_{d-1}(K, \cdot)$ stands for the area measure of $K$. Here,
simplicity means that $\phi$ vanishes on all lower-dimensional sets.

In this paper, we consider valuations taking values in the group
$\mathbb{Z}$ of integers under addition. Clearly, the only continuous
valuations with values in $\Z$ are multiples of the Euler
characteristic
\begin{displaymath}
  \mchi(K)=
  \begin{cases}
    1, & K\neq\emptyset,\\
    0, & K=\emptyset. 
  \end{cases}
\end{displaymath}
It is straightforward to see that a sum of Euler characteristics like
\begin{displaymath}
  \phi(K)=\sum_{i=1}^N \mchi(K\cap C_i)
\end{displaymath}
for convex bodies $C_1,\dots,C_N$ defines an integer-valued monotone
$\sigma$-continuous valuation. Due to the intersection operation,
$\phi$ is no longer continuous in the Hausdorff metric. Adding
negative terms to this sum preserves additivity and
$\sigma$-continuity and may still retain the monotonicity property, as
our examples demonstrate.

Our paper focuses on integer-valued monotone $\sigma$-continuous
valuations without imposing \emph{any invariance assumptions} and
provides their complete characterization in dimensions $1$ and $2$. 
In the main results, we establish that each integer-valued, monotone,
and $\sigma$-continuous valuation in dimensions $1$ and $2$ can be
represented as an at most countable sum of Euler characteristics with
weights $\pm 1$. The convex bodies $C_i$ necessarily form a locally
finite family, and the bodies appearing in the negative terms satisfy a
strict admissibility property with respect to the positive ones. In
other words, each integer-valued monotone $\sigma$-continuous
valuation corresponds to a locally finite integer-valued measure on
the family of convex bodies.

A key step in proving the representation involves the support $F$ of a
valuation $\phi$, which is the set of points $x$ such that
$\phi(\{x\}) \geq 1$. We show that each integer-valued
$\sigma$-continuous valuation is uniquely determined by its values on
singletons and that the intersection of $F$ with any convex body
is polyconvex. For the latter, we apply Eggleston's theorem, which
links polyconvexity to the structure of invisible points. The absence
of such a result in dimensions $3$ and higher makes it impossible to
generalize our technique beyond the planar case.

The main result in dimension 2 is proved in
Section~\ref{sec:planar}. In Section~\ref{sec:valuations-line} we
characterise all real-valued valuations on the line. 
In Section~\ref{sec:mult-count}, we introduce countably generated
valuations, which generalize the weighted sums of Euler
characteristics discussed above. We then define the multiplication of
such valuations by arbitrary $\sigma$-continuous ones and examine the
properties of this product. Section~\ref{sec:oproblems} contains a
collection of open problems and conjectures.

\section{Preliminaries on integer-valued valuations}
\label{sec:prel-integ-valu}

  A set
function $\phi:\Kd\to\R$ is called \emph{monotone} if $\phi(K)\le\phi(L)$
whenever $K\subset L$. In particular, this implies nonnegativity:
$\phi(K) \geq \phi(\emptyset) = 0$.

Following \cite[p.~338]{S14}, we call a valuation $\phi$
\emph{$\sigma$-continuous} if
\begin{equation}
  \label{eq:sc}
  \phi(K)=\lim_{n\to\infty}\phi(K_n)
\end{equation}
for any sequence $(K_n)$ of convex bodies such that $K_n\downarrow K$.
First of all, we note that, for integer-valued valuations, it suffices
to check $\sigma$-continuity only at singletons.

\begin{proposition}
  \label{prp:sc}
  Let $\phi$ be an integer-valued valuation on $\Kd$ such that
  \eqref{eq:sc} holds for all $K=\{x\}$, $x\in\Rd$. Then $\phi$ is
  $\sigma$-continuous.
\end{proposition}
\begin{proof}
  Assume that the claim is false, and \eqref{eq:sc} fails to hold for
  some $K_n\downarrow K$. Then $\phi(K_n)\ne\phi(K)$ for infinitely
  many $n$. For all such $n$ and for any hyperplane $H_1$ meeting $K$
  and dividing $\Rd$ into closed half-spaces $H_1^-$ and $H_1^+$, we
  have
  \begin{align*}
    \phi(K)&=\phi(K\cap H_1^-)+\phi(K\cap H_1^+)-\phi(K\cap H_1),\\
    \phi(K_n)&=\phi(K_n\cap H_1^-)+\phi(K_n\cap H_1^+)-
    \phi(K_n\cap H_1).
  \end{align*}
  Hence, there is $H_1^\bullet\in\{H_1^-,H_1^+,H_1\}$ such that
  $\phi(K\cap H_1^\bullet)\ne\phi(K_n\cap H_1^\bullet)$ for infinitely
  many $n$. Proceeding with this division process and choosing $H_m$
  and $H_m^\bullet$ at the $m$-th step in such a way that
  $K\cap H_1^\bullet\cap\ldots\cap H_m^\bullet$ shrink to a singleton
  $\{x\}$ as $m\to\infty$, we obtain
  \begin{equation}
    \label{eq:phinephi}
    \phi(K\cap H_1^\bullet\cap\ldots\cap H_m^\bullet)\ne
    \phi(K_n\cap H_1^\bullet\cap\ldots\cap H_m^\bullet)
  \end{equation}
  for each fixed $m$ and infinitely many $n$. However, due to the
  $\sigma$-continuity of $\phi$ at $\{x\}$, both sides of
  \eqref{eq:phinephi} converge to $\phi(\{x\})$ as $m,n\to\infty$
  simultaneously. Since both sides are integer-valued, this
  contradicts \eqref{eq:phinephi}.
\end{proof}

Denote $H_{u,t}^-=\{x\in\Rd\colon\langle u,x\rangle\le t\}$. The
following criterion is useful for verifying the monotonicity of a (not
necessarily additive) $\sigma$-continuous set function.

\begin{proposition}
  \label{prp:m}
  A non-negative $\sigma$-continuous set function $\phi$ on $\Kd$ is
  monotone if and only if $\phi(H_{u,t}^-\cap M)$ is non-decreasing in
  $t$ for each fixed $u\in\Sd$ and $M\in\Kd$.
\end{proposition}
\begin{proof}
  The necessity is clear. To prove sufficiency, let $K \subset L$,
  choose an $x_1\in\partial K$, and draw through $x_1$ a supporting
  hyperplane $H_1$ to $K$. Denote by $L_1$ the part of $L$ cut off by
  $H_1$ and containing $K$. Using the assumption with $u$ orthogonal
  to $H_1$ and $M=L$, we get $\phi(L)\ge\phi(L_1)$. Proceeding with
  this process and choosing $x_n\in\partial K$ and $H_n$ at each step
  in such a way that $L_n\downarrow K$, we obtain, by applying the
  assumption to $M =L_{n-1}$, that $\phi(L_{n-1})\ge \phi(L_n)$. Thus,
  $\phi(L)\ge\phi(L_n)$, and, by $\sigma$-continuity,
  $\phi(L)\ge\phi(K)$.
\end{proof}

Note that, as follows from the proof, this proposition remains valid
even if $\sigma$-continuity is replaced by a significantly weaker
condition $\phi(K)\le\sup_{n\ge1}\phi(K_n)$ for any $K_n\downarrow K$.

We now give a somewhat unexpected property of integer-valued
$\sigma$-continuous valuations, which plays a fundamental role in
what follows.

\begin{proposition}
  \label{prp:comp}
  Let $\phi$ and $\phi'$ be integer-valued $\sigma$-continuous
  valuations on $\Kd$ that coincide on singletons:
  $\phi(\{x\})=\phi'(\{x\})$ for any $x\in\Rd$. Then $\phi=\phi'$.
\end{proposition}
\begin{proof}
  We employ reasoning similar to that used in the proof of
  Proposition~\ref{prp:sc}. Suppose the claim is false and
  $\phi(K)\ne\phi'(K)$ for some $K\in\Kd$. Drawing a hyperplane $H$
  that meets $K$ and denoting the closed half-spaces it cuts off by
  $H^-$ and $H^+$, we have
  \begin{align*}
    \phi(K\cap H^-)+\phi(K\cap H^+)-\phi(K\cap H)&=\phi(K)\\
    \ne\phi'(K)&=\phi'(K\cap H^-)+\phi'(K\cap H^+)-\phi'(K\cap H).
  \end{align*}
  Thus, $\phi(K_1)\ne\phi'(K_1)$ for some
  $K_1\in\{K\cap H^-,K\cap H^+,K\cap H\}$.  Proceeding with this
  process so that $K_n\downarrow\{x\}$ for some $x\in\Rd$, we have
  $\phi(K_n)\ne\phi'(K_n)$ for all $n$, while, by $\sigma$-continuity,
  \begin{displaymath}
    \lim_{n\to\infty}\phi(K_n)=\phi(\{x\})=\phi'(\{x\})=
    \lim_{n\to\infty}\phi'(K_n).
  \end{displaymath}
  This is impossible due to the integer-valued property of $\phi$ and
  $\phi'$.
\end{proof}

Proposition~\ref{prp:comp} implies that no simple (i.e., vanishing on
lower-dimensional sets) integer-valued $\sigma$-continuous valuations
exist. Any such valuation must vanish on singletons and is therefore
identically zero.
	
\section{The structure of planar integer-valued valuations}
\label{sec:planar}

In this section, we describe the structure of planar integer-valued
monotone $\sigma$-continuous valuations. Recall that the \emph{normal
  cone} to a closed convex set $C$ at a point $x\in C$ is defined by
\begin{equation}
  \label{eq:nc}
  \NC_C(x)=\big\{u\in\Rd\colon\langle u,y-x\rangle\le0
  \text{ for all }y\in C\big\}
\end{equation}
and adopt the convention $\NC_C(x)=\emptyset$ for $x\notin C$. In
particular,
\begin{align}
  \1_{\NC_C(x)}(0)&=\1_C(x),\label{eq:u0}\\
  \1_{\NC_C(x)}(u)&=\1_C(x)\cdot
  \1\{C\cap\Hr_u^+(x)=\emptyset\},\quad u\ne0,\label{eq:une0}
\end{align}
where
\begin{displaymath}
  \Hr_u^+(x)=\{y\in\Rd\colon\langle u,y-x\rangle>0\}.
\end{displaymath}
This means that $\NC_C(x)$ is empty for $x \notin C$, contains only
$0$ for $x \in \inter C$, and is a non-degenerate closed convex cone
for $x \in \partial C$. Also denote $\Ne=\N\cup\{0,\infty\}$.

\begin{definition}~
  \label{def:lfa}
  \begin{enumerate}[(i)]
  \item A family of $N\in\Ne$ closed convex sets $C_n$ is said to be
    \emph{locally finite} if only finitely many of them hit any fixed
    $K\in\Kd$.
  \item A locally finite family $(C_n^-)$ of cardinality $N^-$ is said
    to be \emph{admissible} with respect to a locally finite family
    $(C_n^+)$ of cardinality $N^+$ if
    \begin{equation}
      \label{eq:ncc}
      \sum_{n=1}^{N^-}\1_{\NC_{C_n^-}(x)}(u)\le
      \sum_{n=1}^{N^+}\1_{\NC_{C_n^+}(x)}(u)
    \end{equation}
    for all $x,u\in\Rd$.
  \end{enumerate}
\end{definition}

In particular, \eqref{eq:ncc} implies that
$\bigcup_n C_n^-\subset \bigcup_n C_n^+$ by letting $u=0$ and using
\eqref{eq:u0}. Further\-more, \eqref{eq:ncc} yields that
$\bigcup_n \partial C_n^-\subset\bigcup_n\partial C_n^+$. Otherwise,
for any $x$ violating this inclusion, the right-hand side of
\eqref{eq:ncc} vanishes for all $u\ne 0$, while the left-hand side
does not for some $u$.

The simplest example of an integer-valued monotone $\sigma$-continuous
valuation is provided by the Euler characteristic
\begin{displaymath}
  \mchi(K)=\1\{K\ne\emptyset\},\quad K\in\Kd.
\end{displaymath}
The following theorem provides a complete description of such
valuations for $d=2$.

\begin{theorem}
  \label{thm:1}
  A function $\phi:\Ktwo\to\Z$ is an integer-valued monotone
  $\sigma$-continuous valuation if and only if there exist
  $N^+,N^-\in\Ne$ and two locally finite families of $N^+$ and $N^-$
  nonempty closed convex sets $C_n^+$ and $C_n^-$ with the latter
  being admissible with respect to the former, such that, for any
  $K\in\Ktwo$,
  \begin{equation}
    \label{eq:repr}
    \phi(K)=\sum_{n=1}^{N^+}\mchi\bigl(K\cap C_n^+\bigr)
    -\sum_{n=1}^{N^-}\mchi\bigl(K\cap C_n^-\bigr).
  \end{equation}
  The families $(C_n^+)$ and $(C_n^-)$ are not uniquely determined:
  $(C_n^+),(C_n^-)$ and $(\widetilde C_n^+),(\widetilde C_n^-)$ define
  the same valuation if and only if
  \begin{equation}
    \label{eq:ind}
    \sum_{n=1}^{N^+}\1_{C_n^+}(x)-
    \sum_{n=1}^{N^-}\1_{C_n^-}(x)=
    \sum_{n=1}^{\widetilde N^+}\1_{\widetilde C_n^+}(x)-
    \sum_{n=1}^{\widetilde N^-}\1_{\widetilde C_n^-}(x)
  \end{equation}
  for all $x\in\Rtwo$. 
\end{theorem}

The non-uniqueness of the representation \eqref{eq:repr} is confirmed
by the following example.

\begin{example}
  \label{ex:non-unique}
  Let $C_1$ and $C_2$ be two convex bodies such that $C_1\cup C_2$ is
  convex. Then the valuation $\phi(K)=\mchi(K\cap (C_1\cup C_2))$ can
  be alternatively represented as
  \begin{displaymath}
    \phi(K)=\mchi(K\cap C_1)+\mchi(K\cap C_2)-\mchi(K\cap(C_1\cap
    C_2)), 
  \end{displaymath}
  which holds since $\mchi(K\cap\cdot)$ is a valuation.  One can
  also notice that both sides agree on singletons, so that the
  equality follows from Proposition~\ref{prp:comp}.
\end{example}

The following examples demonstrate that not all monotone valuations
can be constructed using only $C_n^+$.

\begin{example}
  \label{ex:star}
  Let $a$, $b$ and $c$ be segments positioned as shown in
  Figure~\ref{fig:1}(a) with $O$ denoting their intersection point.
  
  \begin{figure}[h]
    \centering
    \subfigure[]{
      \includegraphics[scale=0.15]{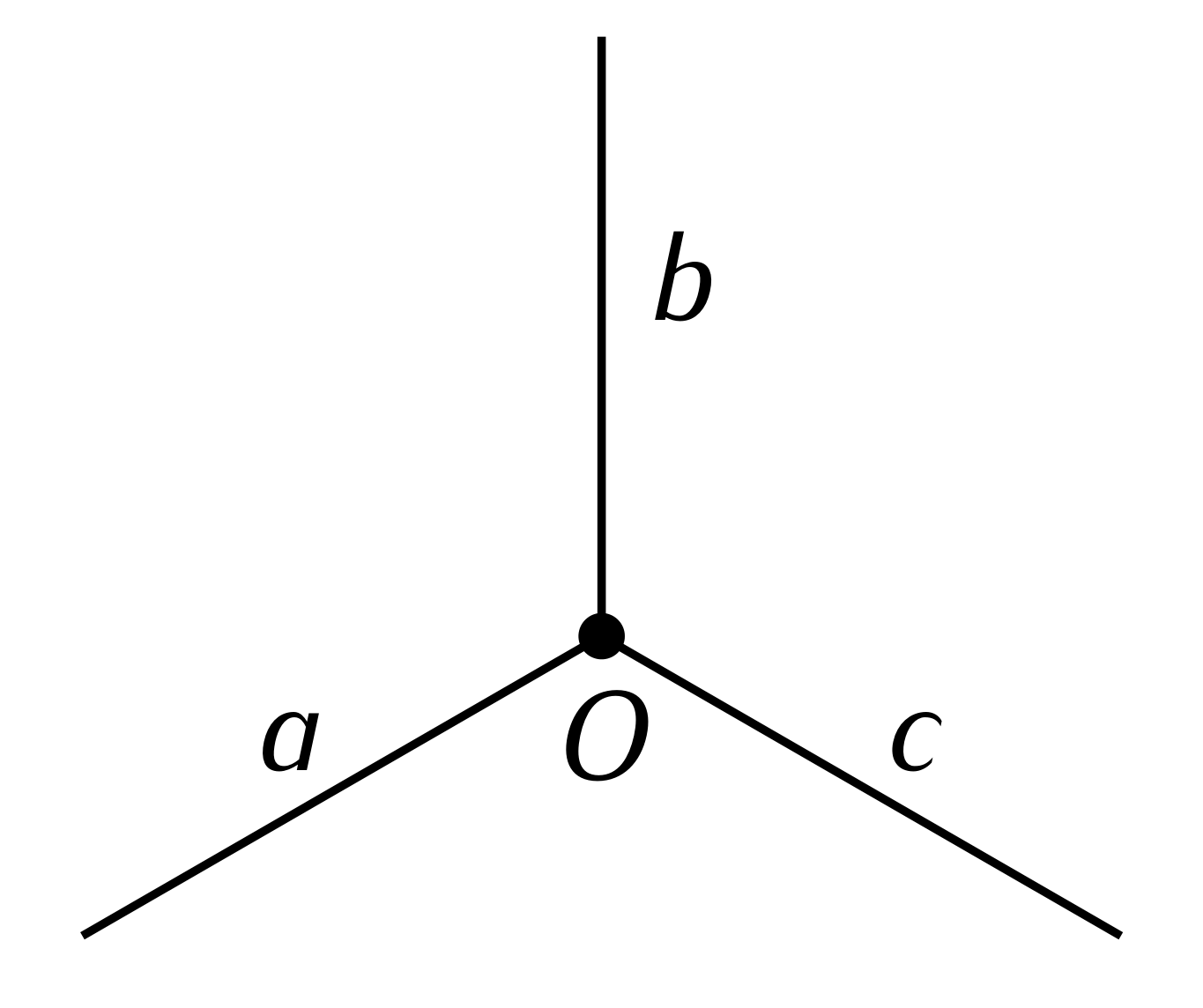}
    }\hspace{60pt}
     \subfigure[]{
      \includegraphics[scale=0.15]{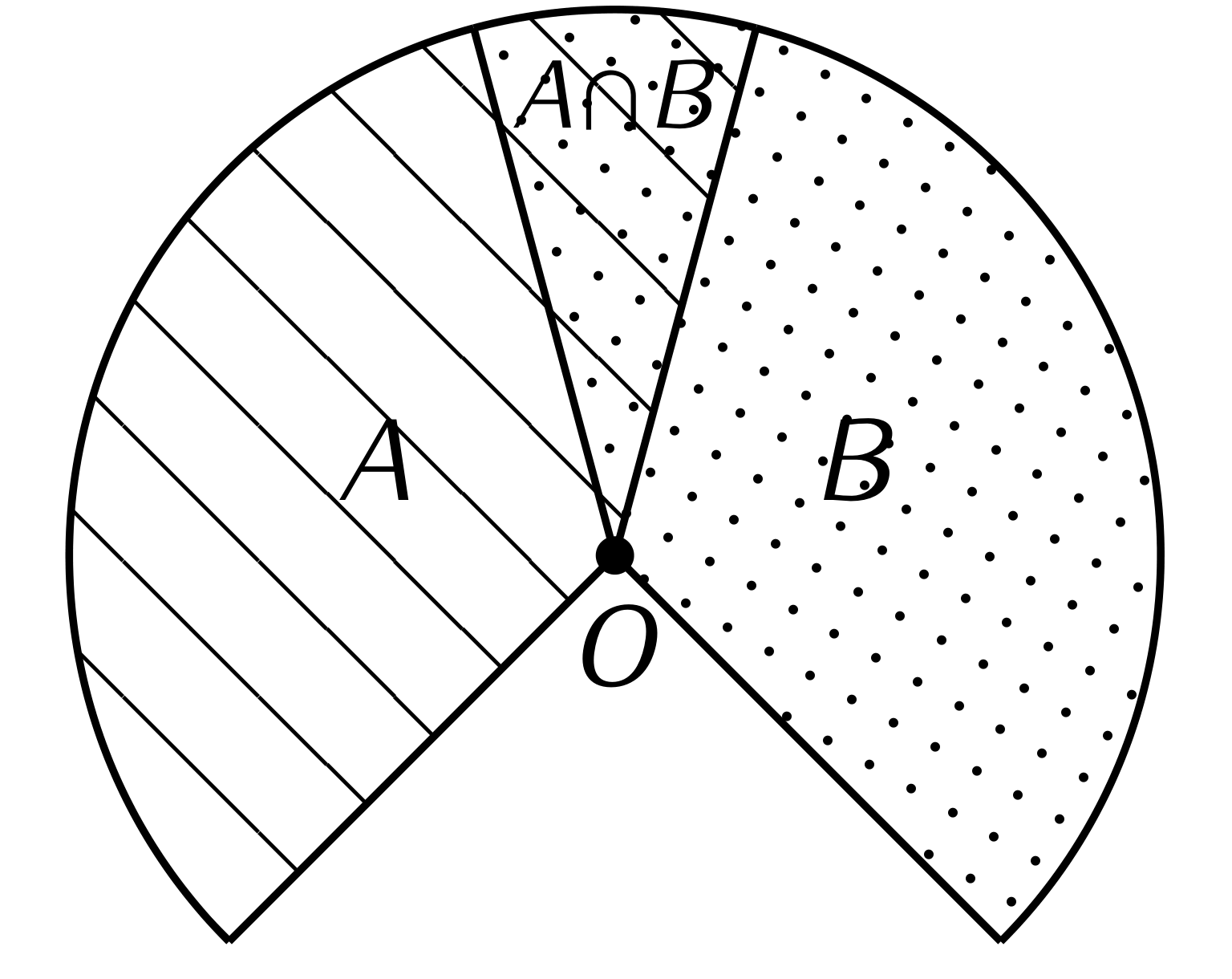}
      }
    \caption{}
    \label{fig:1}
  \end{figure}

  Consider the valuation
  \begin{displaymath}
    \phi(K)=\mchi(K\cap a)+\mchi(K\cap b)+
    \mchi(K\cap c)-\mchi(K\cap\{O\}),\quad K\in\Ktwo.
  \end{displaymath}
  This valuation is integer-valued $\sigma$-continuous and monotone.
  To prove monotonicity, note that $\phi(K)=0$ means that $K$ and
  $a\cup b\cup c$ are disjoint, and $\phi(K)=1$ means that $K$
  intersects exactly one of these segments. Thus, for $\phi(K)\leq 1$
  and $K'\subset K$, we have $\phi(K')\leq\phi(K)$. If $\phi(K)=2$,
  the latter inequality holds because $\phi(K')\leq 2$ for all
  $K'\in\Ktwo$.

  Another way to prove the monotonicity of $\phi$ is to verify the
  admissibility of $(\{O\})$ with respect to $(a,b,c)$. From
  \eqref{eq:nc}, it easily follows that $\NC_{\{O\}}(O)=\R^2$, and for
  any $s\in\{a,b,c\}$, $\NC_s(O)$ is a closed half-plane that does not
  contain $\inter s$, with its boundary passing through $O$ and
  orthogonal to $s$. Since the union of these three half-planes is the
  entire plane, \eqref{eq:ncc} holds for $x=O$. At all other points,
  \eqref{eq:ncc} holds trivially, as its left-hand side vanishes for
  any $u$. Hence, $\phi$ is monotone by Theorem~\ref{thm:1}.
\end{example}

\begin{example}
  Now consider the valuation
  \begin{displaymath}
    \phi=\mchi(\cdot\cap A)+\mchi(\cdot\cap B)+
    \mchi(\cdot\cap\{O\})-\mchi(\cdot\cap A\cap B),
  \end{displaymath}
  see Figure~\ref{fig:1}(b). It is also integer-valued
  $\sigma$-continuous and monotone; both proofs of monotonicity are
  similar to those in Example~\ref{ex:star}. Note that without the
  term $\mchi(\cdot\cap{O})$, the valuation would be
  non-monotone. Altering the positions of the lines bordering
  $A\cap B$ (while maintaining their nonempty intersection) does not
  change the valuation but leads to its different representation.
\end{example}

We will precede the proof of Theorem~\ref{thm:1} with two auxiliary
lemmas. For the first one, we call a point set $P$ an
\emph{invisibility set} if $\phi(\{x\})\ge1$ for each $x\in P$, and,
for any $x,y\in P$, there exists a point $z\in(x,y)$ such that
$\phi(\{z\})=0$. We denote the convex hull of a set $P$ by $\conv P$
and write $\card P$ for the cardinality of $P$. Note that the bound on
the cardinality of $P$ in the following result is apparently far from
optimal one, but it suffices for our purposes.

\begin{lemma}
  \label{lem:invis}
  Let $\phi$ be an integer-valued monotone valuation on $\Ktwo$ and
  $n\in\N$. If $P$ is an invisibility set with
  $\card P\ge 4^n$, then $\phi(\conv P)>\frac n2$. 
\end{lemma}
\begin{proof}
  We proceed by induction on $n$. For $n=1$, the claim is clear, since
  $\phi(\conv P)\geq \phi(\{x\})=1$ for each $x\in P$.  Assume it
  holds for $n-1$. Arguing by contradiction, suppose that
  $\phi(\conv P)\le\frac n2$. Since $\phi$ is $\sigma$-continuous, it
  is always possible to perturb the points of $P$, so that
  $\phi(\conv P)$ does not change and the points of $P$ are in general
  position.

  As before, for a line $H$, we denote by $H^-$ and $H^+$ the two
  closed half-planes into which $H$ divides $\Rtwo$. Draw $H$ in such
  a way that $\card(P\cap H^-)\ge 2^{2n-1}$ and
  $\card(P\cap H^+)\ge 2^{2n-1}$. With a slight adjustment, $H$ can
  always be made to pass through some $x,y\in P$. Mark $z\in(x,y)$
  with $\phi(\{z\})=0$, and denote by $a,b$ the intersection points of
  $H$ and $\partial\conv P$. Connect $z$ by line segments to some
  $u,v\in\partial\conv P$ in such a way as to divide
  $(\conv P)\cap H^-$ and
  $(\conv P)\cap H^+$ into four closed convex polygons
  $Q^{--},Q^{-+},Q^{+-},Q^{++}$ with $\card Q^{ij}\ge 4^{n-1}$,
  $i,j\in\{-,+\}$, see Figure~\ref{fig:points} for $n=2$.

  \begin{figure}[h]
    \centering \includegraphics[scale=0.35]{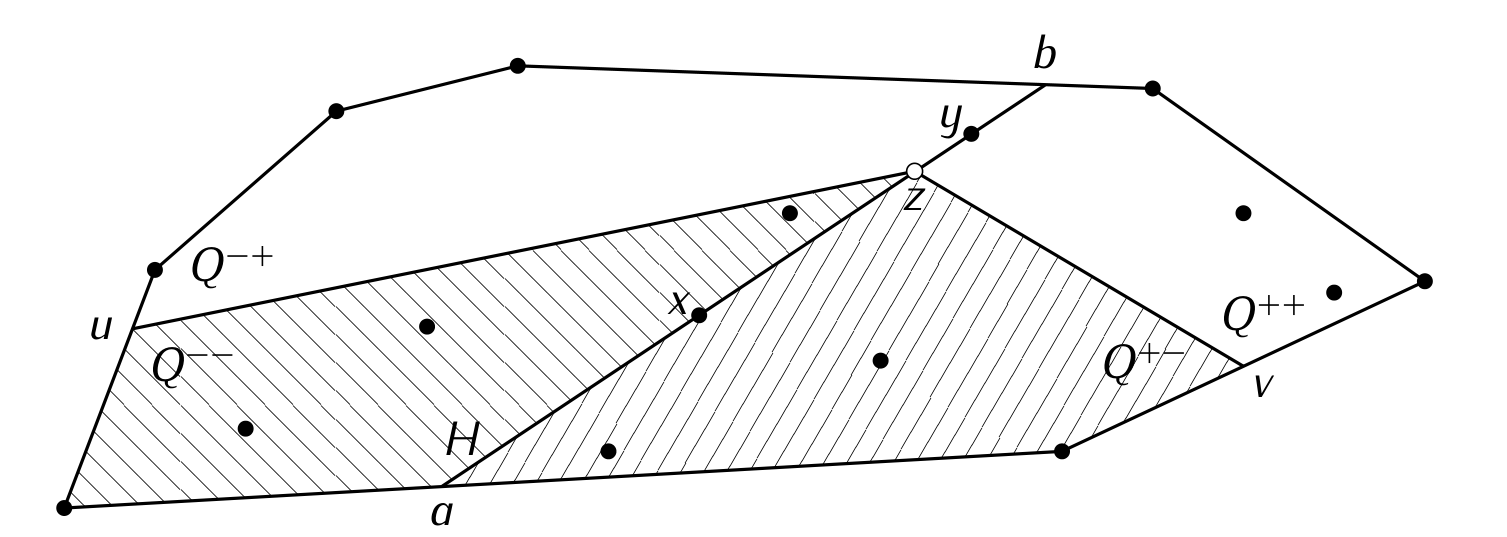}
    \caption{}
    \label{fig:points}
  \end{figure}

  One of the polygons $Q^{--}\cup Q^{+-}$ or $Q^{-+}\cup Q^{++}$ is
  convex, depending on whether the angle $\angle uzv$ is $\le\pi$ or
  $\ge\pi$. Assume the former. Since
  \begin{displaymath}
    \phi([a,z])=\phi([a,b])+\phi(\{z\})-\phi([z,b])
    \le\phi(\conv P)-\phi(\{y\})\le\frac n2-1,
  \end{displaymath}
  applying the induction hypothesis to $Q^{--}$ and $Q^{+-}$ yields
  the contradiction:
  \begin{multline*}
    \frac n2\ge\phi(\conv P)\ge\phi(Q^{--}\cup Q^{+-})\\=
    \phi(Q^{--})+\phi(Q^{+-})-\phi([a,z])>2\,\frac{n-1}2-
    \left(\frac n2-1\right)=\frac n2.\qedhere
  \end{multline*}
\end{proof}

Recall that a set is said to be \emph{polyconvex} if it is a finite
union of (not necessarily disjoint) convex sets, which are called
\emph{convex components}. In particular, the
empty set is also considered as polyconvex.

\begin{lemma}
  \label{lem:polyc}
  Let $\phi$ be an integer-valued monotone $\sigma$-continuous
  valuation defined on closed convex subsets of some $W\in\Ktwo$. Then
  its support
  \begin{displaymath}
    F=\{x\in W\colon\phi(\{x\})\ge1\}
  \end{displaymath}
  is polyconvex, and all its convex components are closed.
\end{lemma}

Before proceeding to the proof, we recall a fact from convex
geometry. For $m\ge2$, a set $S\subset\Rtwo$ is called
\emph{$m$-convex} if, for any $m$ distinct points in $S$, at least one
of the line segments connecting them lies in $S$. In particular,
$2$-convex sets are just convex. According to Eggleston’s theorem
\cite{E74}, a closed $m$-convex set is polyconvex. Note that an
extensive literature has been devoted to deriving upper bounds on the
number of convex components, see \cite{BK76}, \cite{MV99},
\cite{PS90}, etc.

\begin{proof}[Proof of Lemma \ref{lem:polyc}]
  We first note that $F$ is closed. Indeed, if $F\ni x_k\to x$, then,
  for some closed convex neighbourhood $V_x$ of $x$ and some $k\ge1$,
  we have by $\sigma$-continuity and monotonicity that
  \begin{displaymath}
    \phi(\{x\})=\phi(V_x)\ge\phi(\{x_k\})\ge1.
  \end{displaymath}
  This implies $x\in F$.

  Take any set $P$ of $m=4^{2\phi(W)}$ points from $F$. At least one
  of the line segments connecting them lies entirely in $F$:
  otherwise, they would form an invisibility set, and by
  Lemma~\ref{lem:invis} applied to the valuation
  $\phi(W\cap\cdot)$, we would arrive at the contradiction
  $\phi(W)\ge\phi(\conv P)>\phi(W)$. Hence, by Eggleston’s theorem,
  $F=\bigcup_{i=1}^l K_i$ for some $l\ge0$ and convex $K_i$. Taking
  the closures of both sides of this equality and recalling that $F$
  is closed, we arrive at the desired representation.
\end{proof}

\begin{proof}[Proof of Theorem \ref{thm:1}]~\\[-10pt]

  \noindent
  \emph{Sufficiency.} The set function $\phi$ given by \eqref{eq:repr}
  is an integer-valued $\sigma$-continuous valuation, since it is a
  sum of such valuations and, due to local finiteness, this sum has
  only finitely many nonzero terms for each $K$. The only thing that
  remains to be proved is its monotonicity.

  Taking $u=0$ in \eqref{eq:ncc} and using \eqref{eq:u0}, we
  have
  \begin{equation}
    \label{eq:nns}
    \phi(\{x\})=\sum_{n=1}^{N^+}\1_{C_n^+}(x)-
    \sum_{n=1}^{N^-}\1_{C_n^-}(x)\ge0,\quad x\in\R^2.
  \end{equation}
  We will now show that $K\subset L$ implies $\phi(K)\le\phi(L)$. In
  particular, combined with \eqref{eq:nns}, this ensures
  $\phi(K)\ge 0$ for any $K$.

  Denote by $\F_0$ the family of all sets $C_n^+$ and $C_n^-$ which
  appear in \eqref{eq:repr}. 
  Fix $K_0=K\subset L$, and define $\F_1$ to be the family of all sets
  from $\F_0$ that hit $L$ while missing $K_0$. Due to local
  finiteness, $\F_1$ is finite, and it is possible to find a $\delta>0$
  such that the family of sets from $\F_0$ which hit
  $L+B_\delta(0)$ while missing $K_0$ is exactly $\F_1$. Here $+$
  stands for the Minkowski addition. Now replace all sets $C_n^{\pm}$
  from the family $\F_1$ by their intersections with $L$. This does
  not affect the values of $\phi$ on $L$ and its subsets.  

  We claim that there exist
  \begin{enumerate}[1)]
  \item a point $x_1\in L$ on the boundary of some set from $\F_1$,
  \item a supporting line $H_1$ at $x_1$ to this set that separates
    its interior from $K_0$,
  \item a segment $S_{\eps_1}$ on $H_1$ of small length $2\eps_1$ with
    $\eps_1<\delta/\card\F_1$ centered at $x_1$, such that
    $S_{\eps_1}$ hits the same sets from $\F_0$ as $\{x_1\}$ and
    $\conv(K_0\cup S_{\eps_1})\setminus S_{\eps_1}$ does not intersect
    any set from $\F_1$,
  \end{enumerate}
  see Figure~\ref{fig:K01}, where, for simplicity, the sets $C_n^+$
  and $C_n^-$, shown in gray, are depicted as disjoint. The above
  construction can be carried out by choosing $x_1$ to be the
  minimum point of the function $x\mapsto \inf_{z\in K_0}\|z-x\|$ with
  the minimal value $r$ over
  all points $x$ from the union of all sets in the family $\F_1$. In the
  case of multiple minimizers, any of them can be chosen.  Note as
  well that this minimizer may belong to several other sets, say
  $C_{i_1},\dots,C_{i_p}$, from $\F_1$. The $r$-parallel set $K_0^r$
  is smooth at $x$, so there is a unique supporting line which then
  becomes $H_1$. Since any other set from $\F_1$ is further away from
  $K_0$ than $r$, none of them intersects $\conv(K_0\cup S_{\eps_1})$
  for a sufficiently small segment $S_{\eps_1}$ on $H_1$ centered at
  $x_1$. Furthermore, since $\partial K_0^r$ is smooth at $x_1$, no
  set from $C_{i_1},\dots,C_{i_p}$ intersects
  $\conv(K_0\cup S_{\eps_1})\setminus S_{\eps_1}$.

  \begin{figure}[h]
    \centering
    \includegraphics[scale=0.2]{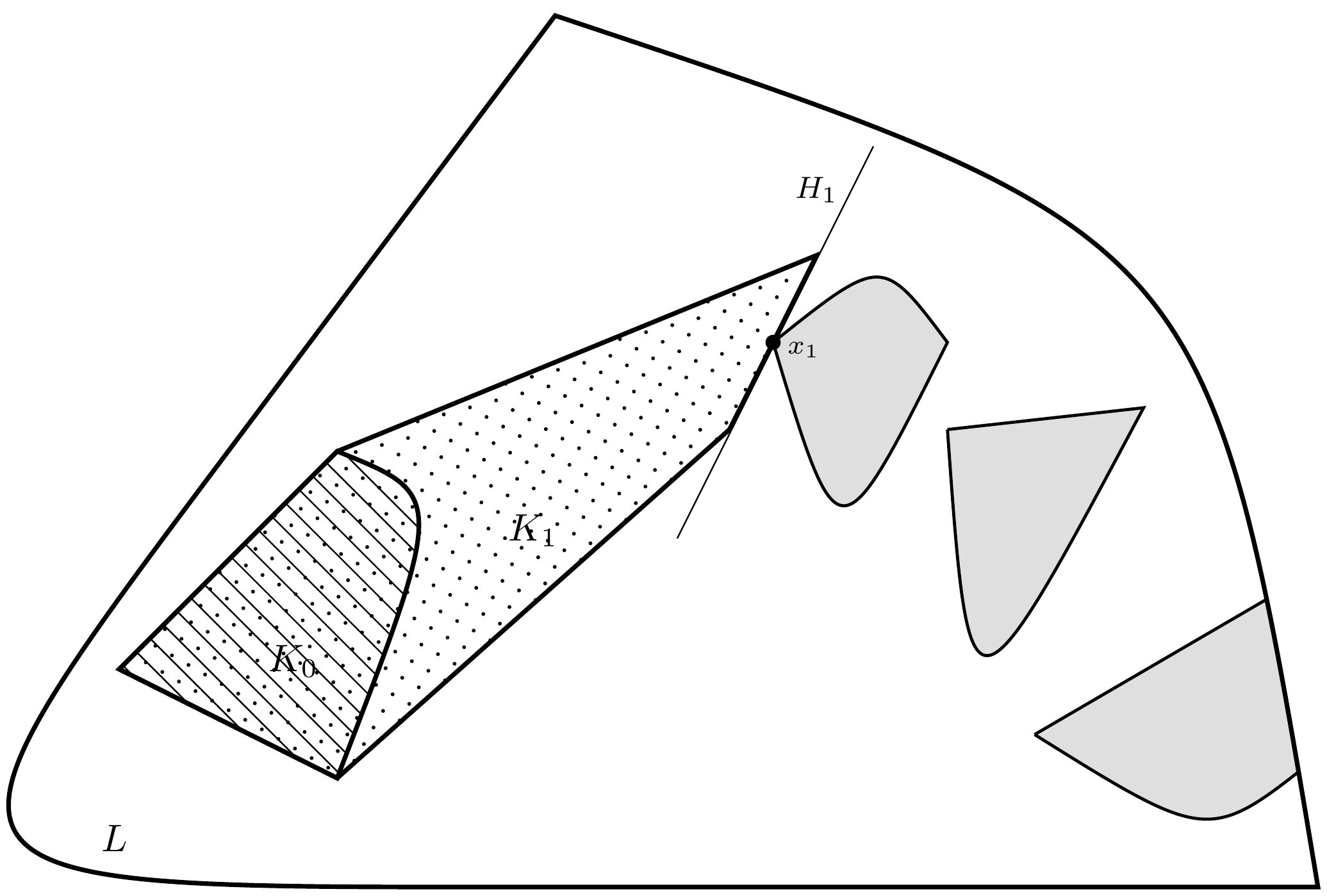}
    \caption{}
    \label{fig:K01}
  \end{figure}

  Denote $K_1=\conv(K_0\cup S_{\eps_1})$
  and observe that $K_0\subset K_1\subset L+B_{\eps_1}(0)$.  
  Let $\F_2$ be the family of
  sets from $\F_0$ (actually, from $\F_1$) that hit $L$ while missing $K_1$.
  Now repeat the above process first with $K_1,\F_2,x_2,H_2,\eps_2$
  instead of $K_0,\F_1,x_1,H_1,\eps_1$, and, similarly, at the
  subsequent steps. This process terminates at step
  $m\leq \card\F_1$, when $\F_{m+1}=\emptyset$. Then
  $K_m\subset L+B_\eps(0)$ with $\eps=\sum \eps_i<\delta$.
  Then the sets $K_m$, $L$, and
  $L+B_\eps(0)$ hit the same sets from the collection $\F_0$.
  Hence, $\phi(K_m)=\phi(L)$, and, to prove monotonicity,
  it remains to show that $\phi(K_{i-1})\le\phi(K_i)$ for each
  $i=1,\ldots,m$.

  Since $\conv(K_{i-1} \cup S_{\eps_i}) \setminus S_{\eps_i}$ does not
  hit any set from $\F_i$, the increment of $\phi$ between $K_{i-1}$
  and $K_i$ is determined exclusively by those $C_n^+$ and $C_n^-$
  that hit $S_{\eps_i}$ but miss the open half-plane
  $\Hr_{u_i}^+(x_i)$ bounded by $H_i$ and containing $K_{i-1}$. By the
  choice of $\eps_i$, such a set hits $S_{\eps_i}$ if and only if it
  contains $x_i$. Hence,
  \begin{align*}
    \phi(K_i)-\phi(K_{i-1})=
    \sum_{n=1}^{N^+}\1_{C_n^+}(x_i)
    &\cdot \1\{C_n^+\cap\Hr_{u_i}^+(x_i)=\emptyset\}\\
    &-\sum_{n=1}^{N^-}\1_{C_n^-}(x_i)\cdot
      \1\{C_n^-\cap\Hr_{u_i}^+(x_i)=\emptyset\},
  \end{align*}
  which is non-negative by \eqref{eq:une0} and \eqref{eq:ncc}.\\[-10pt]

  \noindent 
  \emph{Necessity.}  We first prove that \eqref{eq:repr} holds with
  some locally finite families $(C_n^+)$ and $(C_n^-)$ and afterwards
  address the admissibility of $(C_n^-)$ with respect to $(C_n^+)$. To
  begin with, assume that $\phi$ is supported by a subset of a fixed set
  $W\in\Ktwo$. By monotonicity,
  \begin{displaymath}
    M_{\phi}=\sup_{x\in W}\phi(\{x\})\le\phi(W)<\infty.
  \end{displaymath}
  We will proceed by induction on $M_{\phi}$.

  If $M_{\phi}=0$, then, comparing $\phi$ with the zero valuation
  using Proposition~\ref{prp:comp}, we get $\phi=0$, so that the claim
  holds with $N^+=N^-=0$.

  Now let $M_{\phi}=k$, $k\ge1$, and suppose the claim has been
  established for any valuation $\phi'$ on $W$ such that
  $M_{\phi'}\le k-1$. By Lemma~\ref{lem:polyc}, the support of $\phi$
  is $F=\bigcup_{i=1}^lK_i$ for some $l\ge1$ and closed convex sets
  $K_1,\dots,K_l\subset W$. Denote by
  $\phi\arrowvert_L=\phi(\cdot\cap L)$ the restriction of $\phi$ to
  $L\in\Ktwo$, and consider the valuation
  \begin{equation}
    \label{eq:phiast}
    \phi^\ast=\mchi\!\arrowvert_F+\sum_{r=1}^l (-1)^{r-1}
    \sum_{1\le i_1<\ldots<i_r\le l}
    (\phi-\mchi)\arrowvert_{K_{i_1}\cap\ldots\cap K_{i_r}}.
  \end{equation}
  The valuations $\phi$ and $\phi^\ast$ coincide on singletons: if
  $x\in F$ belongs to exactly $m$ sets from $K_1,\ldots,K_l$ and
  $\phi(\{x\})=p$, then
  \begin{displaymath}
    \phi^\ast(\{x\})=1+(p-1)\sum_{r=1}^m(-1)^{r-1}\binom mr=p.
  \end{displaymath}
  Hence, by Proposition~\ref{prp:comp}, $\phi=\phi^\ast$.  The
  valuation
  $\phi'=(\phi-\mchi)\arrowvert_{K_{i_1}\cap\ldots\cap K_{i_r}}$ is
  integer-valued monotone $\sigma$-continuous, and
  $M_{\phi'}\le k-1$. Thus, by the induction hypothesis, it is of the
  required form. Substituting the expression \eqref{eq:repr} for
  $\phi'$ into \eqref{eq:phiast} yields the required representation
  for $\phi^\ast=\phi$.

  Now consider an integer-valued monotone $\sigma$-continuous
  valuation $\phi$ on the entire $\Ktwo$. For $i,j\in\Z$, denote
  \begin{displaymath}
    \begin{aligned}
      Q_{i,j} & = [i,i+1]\times[j,j+1],
      & \qquad E_{i,j} & = [i,i+1]\times\{j\}, \\
      E'_{i,j} & = \{i\}\times[j,j+1],
      & \qquad V_{i,j} & = \{i\}\times\{j\}.
    \end{aligned}
  \end{displaymath}
  Note that $\Rtwo=\bigcup_{i,j\in\Z} Q_{i,j}$, and double, triple,
  and quadruple intersections of distinct components take the form of
  $E_{i,j}$, $E'_{i,j}$, or $V_{i,j}$, while the intersections of
  higher orders are empty. The restriction of $\phi$ to
  any of these sets is an integer-valued monotone $\sigma$-continuous
  valuation as well. Hence, by the reasoning above, these restrictions
  are of the required form. Applying an analogue of \eqref{eq:phiast}
  to the countable collection of sets $Q_{i,j}$ with intersections
  beyond the fourth order being empty, we obtain the required form of
  $\phi$.

  We now prove that $(C_n^-)$ is admissible with respect to
  $(C_n^+)$. Since $\phi(\{x\})\ge0$ for any $x$, we have
  \begin{displaymath}
    p^-=\sum_{n=1}^{N^-}\1_{C_n^-}(x)\le \sum_{n=1}^{N^+}\1_{C_n^+}(x)=p^+,
  \end{displaymath}
  which, by \eqref{eq:u0}, implies \eqref{eq:ncc} for $u=0$.  Fix
  some $x$ and $u\ne0$, and let
  \begin{equation}
    \label{eq:pq}
    q^+=\sum_{n=1}^{N^+}\1_{C_n^+}(x)\cdot\1\{C_n^+\cap \Hr_u^+(x)=\emptyset\},\qquad
    q^-=\sum_{n=1}^{N^-}\1_{C_n^-}(x)\cdot\1\{C_n^-\cap \Hr_u^+(x)=\emptyset\}.	
  \end{equation} 
  If \eqref{eq:ncc} is violated for $x$ and $u$, then by
  \eqref{eq:une0}, we have $q^+-q^-<0$. It follows from \eqref{eq:pq}
  that
  \begin{equation}
  \begin{multlined}
    \label{eq:pqphi}
    \sum_{n=1}^{N^+}\1_{C_n^+}(x)\cdot\1\{C_n^+\cap
    \Hr_u^+(x)\ne\emptyset\}
    -\sum_{n=1}^{N^-}\1_{C_n^-}(x)
    \cdot\1\{C_n^-\cap \Hr_u^+(x)\ne\emptyset\}\\=(p^+-q^+)-(p^--q^-)=
    (p^+-p^-)-(q^+-q^-)>(p^+-p^-)=\phi(\{x\}).
  \end{multlined}
  \end{equation}
  Due to local finiteness, there are a disk $B_\eps(x)$ that
  hits the same $C_n^+$ and $C_n^-$ as $\{x\}$ and a closed
  convex set $K$, approximating $B_\eps(x)\cap\Hr_u^+(x)$ from the
  inside, that hits the same $C_n^+$ and $C_n^-$ as
  $B_\eps(x)\cap\Hr_u^+(x)$. Hence, the left-hand side of
  \eqref{eq:pqphi} is $\phi(K)$, while the right-hand side is
  $\phi(B_\eps(x))$, which contradicts monotonicity.

  To prove the final claim of the theorem, it suffices to note that
  \eqref{eq:ind} means the equality of the corresponding valuations on
  singletons. By Proposition~\ref{prp:comp}, this implies their
  overall equality.
\end{proof}

\begin{remark}
\label{rem:compact}
Note that, in fact, we constructed the representation \eqref{eq:repr}
with components $C_n^+$ and $C_n^-$ that are not only closed and
convex but also bounded, meaning they belong to $\Ktwo$. However,
using unbounded components is often convenient. For example, for the
Euler characteristic $\mchi$, we can simply take $N^+=1$, $N^-=0$ and
$C_1^+=\Rtwo$.
\end{remark}

The proof that the admissibility of $(C_n^-)$ with respect to
$(C_n^+)$ is both necessary and sufficient for the monotonicity of
$\phi$ extends to any dimension along the same lines. In other words,
if a valuation $\phi$ on $\Kd$ has the form \eqref{eq:repr} with some
locally finite families $(C_n^+)$ and $(C_n^-)$, then it is monotone
if and only if $(C_n^-)$ is admissible with respect to $(C_n^+)$ in
the sense of Definition~\ref{def:lfa}\hs(ii). However, the necessity
of the representation \eqref{eq:repr} beyond the two-dimensional
setting remains an open question: the most critical part of the proof
relies on an application of Eggleston's theorem, and little is known
about its validity in higher dimensions.

The $\sigma$-continuity condition imposed in Theorem~\ref{thm:1} is
crucial. If it is omitted, the class of integer-valued monotone
valuations expands. This is illustrated by the following examples,
which work in spaces of any dimension.

\begin{example}
  \label{ex:open}
  If $N^-=0$, then the right-hand side of \eqref{eq:repr}, written in
  the form of
  \begin{displaymath}
    \phi=\sum_{n=1}^{N^+}\1\{\cdot\cap C_n^+\ne\emptyset\},
  \end{displaymath}
  defines an integer-valued monotone valuation even if the sets
  $C_n^+$ are not necessarily closed.
\end{example}

\begin{example}
  \label{ex:faces}
  For $u\in\Sd$, denote by $H_u$ the supporting hyperplane of $K$ with
  outer normal $u$ and set $K_u=K\setminus(K\cap H_u)$. Thus, $K_u$ is
  $K$ with one exposed face removed. For $N^+\in\Ne$, a set
  $\{u_n\}\subset\Sd$ and a locally finite set
  $\{x_n\} \subset\mathbb{R}^d$, both of cardinality $N^+$, define
  \begin{equation}
    \label{eq:repr2}
    \phi(K)=\sum_{n=1}^{N^+}\1\{x_n\in K_{u_n}\},
    \quad K\in\Kd.
  \end{equation}
  The monotonicity of \eqref{eq:repr2} is clear. To prove additivity,
  we first note that, for $K,L\in\Kd$ with convex union,
  \begin{equation*}
    (K\cup L)_u=K_u\cup L_u\quad\text{and}
    \quad(K\cap L)_u=K_u\cap L_u.
  \end{equation*}
  The only two non-trivial inclusions here are the direct one in the
  first equality and the reverse one in the second. Let $H_{u,x}^+$
  stand for the open half-space with inner normal $u$ whose boundary
  contains $x$.  If $x\in(K\cup L)_u$, then $x$ belongs to, say, $K$,
  and there exists $y\in H_{u,x}^+\cap(K\cup L)$. If $y\in K$, we have
  $x\in K_u$. If, however, $y\in L$, then, due to convexity of
  $K\cup L$, there exists $z\in[x,y]\cap K\cap L$. If $z=x$, we have
  $x,y\in L$, and thus $x\in L_u$. If $z\ne x$, then $x,z\in K$, and
  so $x\in K_u$. This proves the direct inclusion in the first
  equality.

  Now let $x\in K_u\cap L_u$. Then $x\in K\cap L$ and there exist
  $y_1\in H_{u,x}^+\cap K$, $y_2\in H_{u,x}^+\cap L$. Again, due to
  convexity of $K\cup L$, there is $z\in[y_1,y_2]\cap K\cap L$. Hence,
  $z\in H_{u,x}^+\cap(K\cap L)$, and so $x\in(K\cap L)_u$. This proves
  the reverse inclusion in the second inequality.

  The additivity of each summand in \eqref{eq:repr2} follows from the
  identity
  \begin{align*}
    \1\{x_n\in K_{u_n}\}+\1\{x_n\in L_{u_n}\}
    &=\1\{x_n\in K_{u_n}\cup L_{u_n}\}+
      \1\{x_n\in K_{u_n}\cap L_{u_n}\}\\
    &=\1\{x_n\in(K\cup L)_{u_n}\}+
      \1\{x_n\in(K\cap L)_{u_n}\}.
  \end{align*}
  The general case follows by linearity and, if necessary, by passing
  to the limit.
\end{example}

It is interesting to note that, in the one-dimensional case, all
discontinuous integer-valued monotone valuations are fully
characterized by a combination of these two examples. This follows
from Theorem~\ref{thm:2}\hs(iv) in the next section.  On the other
hand, for $\sigma$-continuous valuations with the monotonicity
condition dropped, the representation \eqref{eq:repr} may also fail
even in the one-dimensional setting, as demonstrated by
Examples~\ref{ex:nm1} and~\ref{ex:nm2} in the next section. This
confirms that the Jordan decomposition does not hold for
integer-valued $\sigma$-continuous valuations.

\section{Valuations on the line}
\label{sec:valuations-line}

In this section, we will explore the structure of valuations on $\K$,
with a focus on integer-valued valuations that possess some additional
properties such as monotonicity or $\sigma$-conti\-nuity. In
particular, it will be shown that, in the one-dimensional analogue of
Theorem~\ref{thm:1}, it is always possible to set $N^-=0$, thus
restricting the right-hand side of \eqref{eq:repr} to positive terms
only. The one-dimensional case is, of course, much simpler than the
planar one, which allows us to provide in the following theorem a
complete characterization of all one-dimensional valuations with
certain properties.

We will use the double angle brackets
$\langle\!\langle p,q\rangle\!\rangle$, $-\infty\le p\le q\le\infty$,
to denote any of the four types of intervals: closed, semi-open, or
open.  If $p=-\infty$ or $q=\infty$, the interval on the corresponding
side can only be open. If $p=q$, then
$\langle\!\langle p,q\rangle\!\rangle=[p,p]=\{p\}$.

\begin{theorem}
  \label{thm:2}
  Let $\phi$ be an arbitrary valuation on $\K=\{[a,b]\colon a\le b\}$.
  Then there exist two unique functions $f,g\colon\R\to\R$ with
  $f(0)=0$ such that $\phi([a,b])=g(b)-f(a)$ for any $a\le
  b$. Conversely, any such pair of functions defines a
  valuation. Moreover,
  \begin{enumerate}[(i)]
  \item $\phi$ is integer-valued if and only if $f$ and $g$ are
    integer-valued;
  \item $\phi$ is monotone if and only if $f$ and $g$ are
    non-decreasing and $f\le g$;
  \item $\phi$ is $\sigma$-continuous if and only if $f$ is
    left-continuous and $g$ is right-continuous;
  \item $\phi$ is integer-valued and monotone if and only if there
    exist $N_1,N_2,N_3\in\Ne$, a locally finite family of $N_1$ intervals
    $\langle\!\langle p_n,q_n\rangle\!\rangle$, and two locally finite
    sets of $N_2$ (resp., $N_3$) points $r_n$ (resp., $s_n$), such
    that, for each $[a,b]\in\K$,
    \begin{equation}
    \begin{aligned}
      \label{eq:phi'}
      \phi([a,b])=&\sum_{n=1}^{N_1}
      \1\bigl\{[a,b]\cap\langle\!\langle
      p_n,q_n\rangle\!\rangle\ne\emptyset\bigr\}\\
      +&\sum_{n=1}^{N_2}\1\{r_n\in(a,b]\}+
      \sum_{n=1}^{N_3}\1\{s_n\in[a,b)\};
    \end{aligned}
    \end{equation}
  \item $\phi$ is integer-valued monotone and $\sigma$-continuous if
    and only if there exist $N\in\Ne$ and a locally finite family of
    $N$ closed intervals $[p_n,q_n]$, such that, for each
    $[a,b]\in\K$,
    \begin{equation*}
      \phi([a,b])=\sum_{n=1}^N
      \1\bigl\{[a,b]\cap[p_n,q_n]\ne\emptyset\bigr\}.
    \end{equation*}
  \end{enumerate}
\end{theorem}

Note that, unlike the terms in the last two sums of \eqref{eq:phi'},
$\1\{t\in(a,b)\}$ is not a valuation: additivity is violated, e.g.,
for $K=[t-1,t]$ and $L=[t,t+1]$. Moreover, \eqref{eq:phi'} can be seen
as a combination of Examples~\ref{ex:open} and~\ref{ex:faces} in the
one-dimensional setting.

\begin{proof}[Proof of Theorem~\ref{thm:2}]
  The difference $g(b)-f(a)$ clearly satisfies additivity and so
  defines a valuation. Conversely, for the valuation $\phi$, define
  \begin{equation}
  \begin{aligned}
    \label{eq:fg}
    f(x)&=
    \begin{cases}
      \phi([0,x])-\phi(\{x\}),&x\ge0,\\\phi(\{0\})-\phi([x,0]),&x<0,
    \end{cases}\\
    g(x)&=
    \begin{cases}
      \phi([0,x]),&x\ge0,\\\phi(\{x\})+\phi(\{0\})-\phi([x,0]),
      &x<0.
    \end{cases}
  \end{aligned}
  \end{equation}
  Then $f(0)=0$ and, for $0\le a\le b$, we have by additivity
  \begin{displaymath}
    \phi([a,b])=\phi([0,b])-\phi([0,a])+\phi(\{a\})=g(b)-f(a).
  \end{displaymath}
  The other two cases, $a\le b<0$ and $a<0\le b$, are treated similarly.

  In (i), the ``only'' part follows from \eqref{eq:fg}, while the
  ``if'' part from $\phi([a,b])=g(b)-f(a)$.  The same equality easily
  yields both parts in (ii) and (iii).

  The ``if'' part in (iv) follows from Examples~\ref{ex:open}
  and~\ref{ex:faces}. We now prove the ``only if'' part in (iv).  Let
  $\phi(\{c\})=m=\min_{x\in\R}\phi(\{x\})$ and
  $\phi'=\phi(\cdot-c)-m$. Then $\phi'$ is an integer-valued monotone
  valuation with $\phi'(\{0\})=0$. Hence, for its functions $f$ and
  $g$, we have $f(0)=g(0)=0$. It follows from the previous claims that
  $f$ and $g$ are non-decreasing step functions with integer jumps,
  and $f\le g$. To each point $x>0$ where $g$ has a left
  discontinuity, i.e., $g(x)-g(x-)\ge1$, we associate a pattern of
  $g(x)-g(x-)$ consecutive identical entries \lq\lq$[x$\rq\rq. In a
  similar manner, handle the right discontinuities of $g$, denoting
  their positions as \lq\lq$(x$\rq\rq, $x\ge0$. Then proceed similarly
  with the left and right discontinuities of $f$, using the notation
  \lq\lq$x)$\rq\rq\ and \lq\lq$x]$\rq\rq, respectively. Finally,
  combine these patterns in increasing order of $x$ into a single, at
  most countable sequence. In the case of patterns with the same $x$,
  they should be arranged in the following order:
  $[x\ldots(x\ldots x)\ldots x]\ldots$. The resulting sequence encodes
  both $f$ and $g$ on $[0,\infty)$.

  For example, for the functions
  \begin{gather*}
    f=2\1_{(0,2)}+3\1_{\{2\}}+5\1_{(2,4]}+7\1_{(4,6]}+
    10\1_{(6,\infty)},\\
    g=3\1_{(0,1]}+4\1_{(1,2]}+5\1_{(2,4)}+6\1_{\{4\}}+
    7\1_{(4,6)}+8\1_{\{6\}}+12\1_{(6,\infty)},
  \end{gather*}
  using this algorithm, we obtain the following sequence:
  \begin{displaymath}
    (0\;(0\;(0\;0]\;0]\;(1\;(2\;2)\;2]\;2]\;[4\;(4\;4]\;4]\;[6\;(6\;
    (6\;(6\;(6\;6]\;6]\;6].
  \end{displaymath}
  Now, for the first opening bracket, find the nearest closing one on
  the right, note the resulting interval, remove the used pair from
  the sequence and repeat the procedure. If there are not enough
  closing brackets, use $\infty)$ as many times as needed. In the
  above example, we arrive at the following set of
  intervals:
  \begin{displaymath}
    (0,0],\;(0,0],\;(0,2),\;(1,2],\;(2,2],\;[4,4],\;(4,4],\;[6,6],\;
    (6,6],\;(6,6],\;(6,\infty),\;(6,\infty).
  \end{displaymath}
  The resulting intervals can be real, such as the four types of
  $\langle\!\langle p,q\rangle\!\rangle$, or virtual, such as $[r,r)$
  and $(s,s]$. A virtual interval $(t,t)$ is impossible by
  construction due to the condition $f\le g$. Along the same lines, a
  similar list of real and virtual intervals can be constructed on
  $(-\infty,0]$.

  Consider the valuation $\phi''$ constructed according to
  \eqref{eq:phi'}, by incorporating the real inter\-vals
  $\langle\!\langle p_n,q_n\rangle\!\rangle$ into the terms of the
  first sum, and the points $r_n$, $s_n$ defining the virtual
  intervals into the terms of the second and third sums. Calculating
  by \eqref{eq:fg} the functions $f_n$ and $g_n$ corresponding to all
  six types of terms in \eqref{eq:phi'}, it is easy to see that the
  step functions $f$ and $g$ for $\phi''$ have the same positions and
  structure of discontinuities as those for $\phi'$. Hence, these
  functions coincide, and so $\phi''=\phi'$.  Thus, $\phi'$ takes the
  form of \eqref{eq:phi'}.  Shifting $\phi'$ to the right by $c$ and
  adding $m=m\1\{[a,b]\cap(-\infty,\infty)\ne\emptyset\}$, we arrive
  at the required representation for $\phi$.

  The ``if'' part in (v) is clear. The ``only if'' part follows from
  (iv) and the fact that all other terms in \eqref{eq:phi'} are easily
  seen not to be $\sigma$-continuous.
\end{proof}

We can now give the examples announced at the end of
Section~\ref{sec:planar}, which demonstrate that, even in the
one-dimensional case, the representation \eqref{eq:repr} may fail if
the monotonicity condition on the valuation is dropped.

\begin{example}
  \label{ex:nm1}
  Let
  \begin{displaymath}
    f=0\quad\text{and}\quad
    g=\sum_{n=1}^\infty\1_{\left[\frac{2n-1}{2n},\frac{2n}{2n+1}\right)}.
  \end{displaymath}
  By Theorem~\ref{thm:2}, $\phi([a,b])=g(b)-f(a)=g(b)$ defines an
  integer-valued $\sigma$-continuous valuation on $\K$. Since the pair
  $(f_{p,q},g_{p,q})=(0,\1_{[p,q)})$ corresponds to the
  valuation
  \begin{displaymath}
    \phi_{p,q}([a,b])=\1\{b\in[p,q)\}=\1\{[a,b]\cap[p,q]\ne
    \emptyset\}-\1\{[a,b]\cap\{q\}\ne\emptyset\},
  \end{displaymath}
  we arrive at the representation
  \begin{equation}
    \begin{aligned}
      \label{eq:pm_repr}
      \phi([a,b])=\sum_{n=1}^\infty\1\{[a,b]\cap C_n^+\ne\emptyset\}&-
      \sum_{n=1}^\infty\1\{[a,b]\cap C_n^-\ne\emptyset\}\\&=
      \sum_{n=1}^\infty\mchi\{[a,b]\cap C_n^+\}-
      \sum_{n=1}^\infty\mchi\{[a,b]\cap C_n^-\},
    \end{aligned}
  \end{equation}
  where $C_n^+=\left[\frac{2n-1}{2n},\frac{2n}{2n+1}\right]$,
  $C_n^-=\left\{\frac{2n}{2n+1}\right\}$, and $\infty-\infty=0$ by
  convention. The families $(C_n^+)$ and $(C_n^-)$ are not locally
  finite.
\end{example}

\begin{example}
  \label{ex:nm2}
  Let $f=0$ and
  $g(x)=\bigl\lfloor\frac 1{1-x}\bigr\rfloor\cdot\1_{(-\infty,1)}$,
  $x\in\mathbb R$. Since
  $g=\sum_{n=1}^\infty\1_{\left[\frac{n-1}n,1\right)}$, the above
  reasoning leads to \eqref{eq:pm_repr} with
  $C_n^+=\left[\frac{n-1}n,1\right]$, $C_n^-=\{1\}$ for all $n$, and
  the same convention. This time, the families $(C_n^+)$ and $(C_n^-)$
  are neither locally finite, nor is the sum in \eqref{eq:ind} even
  well defined.
\end{example}

In both of the above examples, there are no other locally finite
families $(\widetilde C_n^+)$ and $(\widetilde C_n^-)$. Indeed,
denoting $y_k=\frac k{k+1}$, we have by \eqref{eq:pm_repr} that
$\phi(\{y_k\})\ne \phi(\{y_{k+1}\})$ for $k\ge1$. Hence, on each
interval $[y_k,y_{k+1}]$, there must be a point from some
$\partial\widetilde C_n^+$ or $\partial\widetilde C_n^-$. This
contradicts the local finiteness.

\section{Multiplication of countably generated
  valuations}
\label{sec:mult-count}

Theorems~\ref{thm:1} and~\ref{thm:2}\hs(v) lead us to the following
general definition.

\begin{definition}
  \label{def:cg}
  A valuation $\phi$ on $\Kd$ is called \emph{countably generated} if
  there exist $N\in\Ne$, a locally finite family of $N$ nonempty
  closed convex sets $C_n$, and a set of $N$ real numbers $\alpha_n$
  such that
  \begin{equation}
    \label{eq:sum-form}
    \phi(K)=\sum_{n=1}^N\alpha_n\mchi
    \bigl(K\cap C_n\bigr),\quad K\in\Kd.
  \end{equation}
\end{definition}

While in Definition~\ref{def:cg} the sets $C_n$ were assumed to be
only closed and convex, an equivalent representation with compact
$C_n$ follows from the inclusion-exclusion argument used in the proof
of Theorem~\ref{thm:1}.

The above theorems show that any integer-valued monotone
$\sigma$-continuous valuation on $\K$ or $\Ktwo$ is countably
generated with all $\alpha_n=1$ if $d=1$ and $\alpha_n=\pm 1$ if
$d=2$.

Any countably generated valuation is clearly $\sigma$-continuous.  Let
$\Vd$ stand for the vector space of all $\sigma$-continuous valuations
on $\Kd$ equipped with the natural operations of addition and
multiplication by real numbers, and denote by $\Gd$ its subspace of
countably generated valuations. Note that elements of $\Gd$ are
completely determined by their values on singletons: if
$\phi,\phi'\in\Gd$ are defined by $N,(\alpha_n),(C_n)$ and
$N',(\alpha'_n),(C'_n)$, respectively, then $\phi=\phi'$ if and only
if 
\begin{equation}
  \label{eq:phi=phi'}
  \sum_{n=1}^N\alpha_n\1_{C_n}=\sum_{n=1}^{N'}\alpha'_n\1_{C'_n}.
\end{equation}
This can be proved along the same lines as Proposition~\ref{prp:comp}.

For a countably generated valuation, \emph{multiplication} by a
$\sigma$-continuous valuation can be defined as follows.  For
$\psi\in\Vd$ and $\phi\in\Gd$ given by \eqref{eq:sum-form}, define
\begin{equation}
  \label{eq:int}
  (\phi\cdot\psi)(K)=\sum_{n=1}^N\alpha_n\psi(K\cap C_n),
  \quad K\in\Kd.
\end{equation}
The terms on the right-hand side are well defined, since
$K\cap C_n\in\Kd$ for all $n$. If $N=\infty$, only a finite number of
them are non-zero due to the local finiteness of $(C_n)$. Finally, the
value of the sum on the right-hand side of \eqref{eq:int} does not
depend on the specific choice of $N,(\alpha_n),(C_n)$ in the
representation of $\phi$ by \eqref{eq:sum-form}. Indeed, this sum is
the Groemer integral of $\sum_{n=1}^N\alpha_n\1_{K\cap C_n}$ with
respect to $\psi$, see \cite{G78}. This integral is well defined for
$\psi\in\Vd$ by Theorem~3 in the same paper. It remains to note that,
for another set $N',(\alpha'_n),(C'_n)$ corresponding to $\phi$, we
have
\begin{displaymath}
  \sum_{n=1}^{N'}\alpha'_n\1_{K\cap C'_n}=
  \1_K\cdot\sum_{n=1}^{N'}\alpha'_n\1_{C'_n}=
  \1_K\cdot\sum_{n=1}^N\alpha_n\1_{C_n}=
  \sum_{n=1}^N\alpha_n\1_{K\cap C_n}
\end{displaymath}
by \eqref{eq:phi=phi'}.

In the following proposition, we list the basic properties of this product.

\begin{proposition}
  \label{prp:int}
  For fixed $K\in\Kd$,
  \begin{enumerate}[(i)]
  \item $(\phi,\psi)\mapsto(\phi\cdot\psi)(K)$ is a bilinear map from
    $\Gd\times\Vd$ to $\R$;
  \item $(\phi\cdot\psi)(K)=(\psi\cdot\phi)(K)$ on $\Gd\times\Gd$;
  \item $(\chi\cdot\psi)(K)=\psi(K)$, where $\chi$ is the Euler
    characteristic.
  \end{enumerate}
  For fixed $\phi\in\Gd$ and $\psi\in\Vd$,
  \begin{enumerate}[(i)]
    \setcounter{enumi}{3}
  \item $(\phi\cdot\psi)(\cdot)$ is a $\sigma$-continuous valuation,
    that is, this operation acts from $\Gd\times\Vd$ into $\Vd$,
    moreover, $(\phi\cdot\psi)(\{x\})=\phi(\{x\})\hs\psi(\{x\})$ for
    each $x\in\Rd$;
  \item if $\psi\in\Gd$, then $\phi\cdot\psi\in\Gd$ as well, more
    precisely, if $\phi$ is defined by $N,(\alpha_n),(C_n)$, and
    $\psi$ by $N',(\alpha'_n),(C'_n)$, then $\phi\cdot\psi$ is defined
    by
    $NN',(\alpha_n{\alpha_m}\hspace{-5pt}'\hspace{3pt}),
    (C_n\cap{C_m}\hspace{-5pt}'\hspace{3pt})$.
  \end{enumerate}
\end{proposition}
\begin{proof}
  (i) follows directly from \eqref{eq:int}. For
  \begin{equation}
    \label{eq:monom}
    \phi=\1\{\cdot\cap C\ne\emptyset\}\quad\text{and}\quad\psi=
    \1\{\cdot\cap C'\ne\emptyset\},
  \end{equation}
  we have
  \begin{equation}
    \label{eq:phipsi}
    (\phi\cdot\psi)(K)=\psi(K\cap C)=\1\{K\cap C
    \cap C'\ne\emptyset\}=\phi(K\cap C')=(\psi\cdot\phi)(K).
  \end{equation}
  The general case of (ii) follows by linearity. Statement (iii)
  directly results from
  $\chi=\1\{\cdot\cap\Rd\ne\emptyset\}$.
	
  For (iv), if $\phi=\1\{\cdot\cap C\ne\emptyset\}$, then
  $(\phi\cdot\psi)(K)=\psi(K\cap C)$, which is a $\sigma$-continuous
  valuation, then use linearity.  The equality in (iv) follows from
  \begin{displaymath}
    (\phi\cdot\psi)(\{x\})=
    \sum_{n=1}^N\alpha_n\psi(\{x\}\cap C_n)
    =\sum_{n=1}^N\alpha_n\mchi(\{x\}\cap C_n)\hs\psi(\{x\})
    =\phi(\{x\})\hs\psi(\{x\}).
  \end{displaymath}
  For (v), under \eqref{eq:monom}, the result follows from
  \eqref{eq:phipsi}. In the general case, again use linearity.
\end{proof}

The valuation $\phi\cdot\psi$ can be naturally called the
product of $\phi$ and $\psi$ for the following reason.  The
multiplication of smooth valuations introduced by S.~Alesker
\cite{A04} can be, in the translation-invariant case, succinctly
described as follows. Let $\phi_0$ stand for the volume, and
define $\phi_A=\phi_0(\cdot+A)$, $A\in\Kd$, with $+$ being the
Minkowski addition. The Alesker product is defined by setting
\begin{equation}
  \label{eq:phi_prod}
  (\phi_A\cdot\phi_B)(K)=\phi_0\bigl(\Delta(K)+A\times B\bigr),\quad K\in\Kd,
\end{equation}
where $\Delta\colon\Rd\to\Rd\times\Rd$ stands for the diagonal
embedding $x\mapsto(x,x)$. This product then extends by linearity and
continuity to all pairs of smooth translation-invariant valuations by
making use of the density of such valuations as follows from Alesker’s
irreducibility theorem.

Except for the multiples of the Euler characteristic, countably
generated valuations are neither smooth nor
translation-invariant. Therefore, to use this approach, the basic
valuation $\phi_0$ needs to be redefined. Let
$\phi_0=\1\{0\in\cdot\}$. Then
$\phi_A=\phi_0(\cdot+A)=\1\{\cdot\cap(-A)\ne\emptyset\}$ for any (not
necessarily bounded) nonempty closed convex set $A$. It follows from
\eqref{eq:phi_prod} that
\begin{displaymath}
  (\phi_A\cdot\phi_B)(K)
  =\1\bigl\{\Delta(K)\cap\bigl((-A)\times(-B)\bigr)\ne\emptyset\bigr\}
  =\1\{K\cap(-A)\cap(-B)\ne\emptyset\},\quad K\in\Kd,
\end{displaymath}
which is consistent with the description of the product given in
Proposition~\ref{prp:int}\hs(v). This is in line with the
intersectional approach to the Alesker product given in
\cite{F16} within the framework of smooth manifolds; this connection
can be made formal using generalized valuations and wave front sets,
see \cite{MR2905004}.

\section{Open problems}
\label{sec:oproblems}

In this section, we outline some open problems and conjectures. First,
a major issue is to consider the case of general dimensions. This
cannot be done by mimicking the proof of Theorem~\ref{thm:1} due to
the absence of a result relating $m$-convexity and polyconvexity in
dimensions 3 and more.

\begin{openp}
  Characterize integer-valued monotone $\sigma$-continuous valuations
  in dimensions 3 and higher.
\end{openp}

Counterexamples show that it is not possible to obtain meaningful
results for valuations which are not $\sigma$-continuous.  However,
relaxing the monotonicity condition may be interesting also in
dimension 2.

\begin{openp}
  Obtain characterization results under weaker variants of the
  monotonicity condition, e.g., assuming nonnegativity or local
  boundedness of variation in the sense of
  \begin{displaymath}
    \sup_{L\subset K,\,L\in\Kd}\phi(L)\le C_K,\quad K\in\Kd,
  \end{displaymath}  
  where $C_K$ is a constant depending on $K$.
\end{openp}

\begin{openp}
  Which property of an integer-valued monotone $\sigma$-continuous
  valuation $\phi$ ensu\-res that its representation \eqref{eq:repr}
  contains no negative terms? This question can be posed in general
  dimension, assuming that the representation \eqref{eq:repr} holds.
\end{openp}

The representation \eqref{eq:repr} can be interpreted as follows. 
Consider an integer-valued signed measure on the space of closed
convex sets in $\Rtwo$ of the form
\begin{displaymath}
  \mu=\sum_{n=1}^{N^+}\delta_{C_n^+}-
  \sum_{n=1}^{N^-}\delta_{C_n^-},
\end{displaymath}
where $\delta_C$ stands for the unit mass at $C$. By
Remark~\ref{rem:compact}, we may assume that this measure is defined
only on $\Ktwo$. Then \eqref{eq:repr} can be written in the following
integral form
\begin{displaymath}
  \phi(K)=\int_{\Ktwo}\mchi(K\cap C)\,\mu(\mathrm dC),\quad
  K\in\Ktwo.
\end{displaymath}
More generally, by \eqref{eq:sum-form}, any countably generated
valuation on $\Kd$ can be written in the same form with
$\mu=\sum_{n=1}^N\alpha_n\delta_{C_n}$ for real numbers $\alpha_n$.

We call a measure $\mu$ on $\Kd$ (with its Borel $\sigma$-algebra
generated by the Hausdorff metric) \emph{locally finite} if
$\mu(\C_K)<\infty$ for all $K\in\Kd$, where
$\C_K=\{C\in\Kd\colon K\cap C\ne\emptyset\}$. An arbitrary locally
finite signed measure $\mu$ on $\Kd$ yields a valuation by letting
\begin{equation}
  \label{eq:integral}
  \phi(K)=\int_{\Kd}\mchi(K\cap C)\,\mu(\mathrm dC)=
  \int_{\Kd}\1\{K\cap C\ne\emptyset\}\,\mu(\mathrm dC)=\mu(\C_K).
\end{equation}
Since $\C_{K_n} \downarrow \C_K$ as $K_n \downarrow K$, this valuation
is $\sigma$-continuous due to the $\sigma$-additivity of the measure
$\mu$.

\begin{openp}
  Identify $\sigma$-continuous valuations on $\Kd$ such that
  \eqref{eq:integral} holds for a locally finite signed measure $\mu$
  on $\Kd$?  
  Note that, as follows from Example~\ref{ex:non-unique}, such a
  measure need not be unique.
\end{openp}

The set of valuations admitting an integral representation of the form
\eqref{eq:integral} is far from being limited to countably generated
valuations. For instance, the $d$-dimensional volume can be expressed
in this form with a measure $\mu$ concentrated on singletons
\begin{displaymath}
  \mu(\{x\}\colon x\in B)=\lambda_d(B),\quad B\in\mathcal B(\Rd),
\end{displaymath}
where $\lambda_d$ stands for the $d$-dimensional Lebesgue measure.
Similar representations hold for intrinsic volumes.

We conjecture that \eqref{eq:integral} holds for a very broad family
of valuations. Examples~\ref{ex:nm1} and~\ref{ex:nm2} demonstrate that
this does not hold for all $\sigma$-continuous valuations, since the
families $(C_n^+)$ and $(C_n^-)$ in these examples do not satisfy the
local finiteness condition and so the measure $\mu$ is not locally
finite. This may be explained by the lack of monotonicity in these
valuations.

\begin{openp}
  Is the family of countably generated valuations dense (in some sense)
  in the space of all $\sigma$-continuous valuations? 
\end{openp}

%
%


The following problems address changing the range of values and/or the
definition domain of valuations.

\begin{openp}
  Characterize valuations taking values in other semigroups, such as
  $(\mathbb{Z}/n\mathbb{Z},+)$, $(\mathbb{Z}/n\mathbb{Z},\times)$,
  $(\mathbb{Q},+)$, etc.
\end{openp}

\begin{openp}
  Characterize integer-valued valuations on convex functions.  It is
  very likely that this can be done using our methods. 
\end{openp}

\section*{Acknowledgment}

AI was supported by the Swiss National Science Foundation, Grant
No.~229505.\\TV was supported by the Swiss National Science
Foundation, Grant No.~10001553.


\begin{thebibliography}{1}

\bibitem{A04}
S.~Alesker.
\newblock The multiplicative structure on continuous polynomial valuations.
\newblock {\em Geom. Funct. Anal.}, 14(1):1--26, 2004.

\bibitem{MR2905004}
S.~Alesker and A.~Bernig.
\newblock The product on smooth and generalized valuations.
\newblock {\em Amer. J. Math.}, 134(2):507--560, 2012.

\bibitem{BK76}
M.~Breen and D.~C. Kay.
\newblock General decomposition theorems for {$m$}-convex sets in the plane.
\newblock {\em Israel J. Math.}, 24(3-4):217--233, 1976.

\bibitem{E74}
H.~G. Eggleston.
\newblock A condition for a compact plane set to be a union of finitely many
  convex sets.
\newblock {\em Proc. Cambridge Philos. Soc.}, 76:61--66, 1974.

\bibitem{F16}
J.~H.~G. Fu.
\newblock Intersection theory and the {Alesker} product.
\newblock {\em Indiana Univ. Math. J.}, 65(4):1347--1371, 2016.

\bibitem{G78}
H.~Groemer.
\newblock On the extension of additive functionals on classes of convex sets.
\newblock {\em Pacific J. Math.}, 75(2):397--410, 1978.

\bibitem{MV99}
J.~Matou\v{s}ek and P.~Valtr.
\newblock On visibility and covering by convex sets.
\newblock {\em Israel J. Math.}, 113:341--379, 1999.

\bibitem{PS90}
M.~A. Perles and S.~Shelah.
\newblock A closed {$(n+1)$}-convex set in {${\mathbf R}^2$} is a union of
  {$n^6$} convex sets.
\newblock {\em Israel J. Math.}, 70(3):305--312, 1990.

\bibitem{S14}
R.~Schneider.
\newblock {\em Convex Bodies: the {B}runn-{M}inkowski Theory}.
\newblock Cambridge University Press, Cambridge, 2014.

\end{thebibliography}

\end{document}